\documentclass{amsart}
\usepackage{graphicx} 

\usepackage{amsfonts,amsthm,mathtools,amsmath,amssymb,fullpage,enumitem,xcolor,hyperref,array,booktabs,comment}
\usepackage{tikz}
\usetikzlibrary{positioning, arrows.meta}

\newtheorem{theorem}{Theorem}
\newtheorem{proposition}{Proposition}
\newtheorem{corollary}{Corollary}

\newtheorem{hypothesis}{Hypothesis}
\newtheorem{lemma}{Lemma}
\newtheorem{question}{Question}
\theoremstyle{definition}

\newtheorem{remark}[theorem]{Remark}

\newcommand{\tors}{\operatorname{tors}}

\newcommand{\Gal}{\operatorname{Gal}}
\newcommand{\Jac}{\operatorname{Jac}}
\newcommand{\Sym}{\operatorname{Sym}}

\newcommand{\gon}{\operatorname{gon}}
\newcommand{\ord}{\operatorname{ord}}

\newcommand{\Q}{\mathbb{Q}}
\newcommand{\Z}{\mathbb{Z}}
\newcommand{\GL}{\operatorname{GL}}

\newcommand{\SL}{\operatorname{SL}}

\newcommand{\Supp}{\operatorname{Supp}}

\DeclareMathOperator{\im}{im}

\newcommand{\calE}{{\mathcal E}}

\title{On the Finiteness of Isolated $j$-invariants for $X_1(N)$}
\author{Abbey Bourdon}

\begin{document}
\begin{abstract}
    Characterizing isolated points on the modular curve $X_1(N)$ is a key obstruction to classifying all points of a fixed degree. These points do not lie in infinite parameterized families, making them difficult to obtain through geometric constructions. In this paper, we focus on the collection of ``isolated $j$-invariants" for $X_1(N)$, which are the values obtained by mapping isolated points to the $j$-line. Work of the author in collaboration with Ejder, Liu, Odumodu, and Viray \cite{BELOV} asks whether there are only finitely many isolated $j$-invariants lying in extensions of bounded degree. Here, we explore how this question relates to other uniformity problems in the field and give new finiteness results for isolated $j$-invariants in $\Q$. As an application, we show similar methods give sharpened polynomial bounds on torsion for non-CM elliptic curves having rational $j$-invariant.
\end{abstract}

\maketitle

\section{Introduction}
The modular curve $X_1(N)$ is a smooth projective curve over $\Q$ whose (non-cuspidal) points parametrize elliptic curves with a distinguished point of order $N$. A classical problem is to characterize the $K$-rational points of $X_1(N)$ for a fixed number field $K$. For $K=\Q$, this was done by Mazur, confirming a conjecture of Ogg \cite{OggConjecture}. The classification shows $X_1(N)(\Q)$  has a non-cuspidal point if and only if the curve has genus 0. That is, these points are ``explained by geometry." 

\begin{theorem}[Mazur \cite{mazur77}]
    $X_1(N)$ has a non-cuspidal rational point if and only if $1 \leq N \leq 10$ or $N=12$.
\end{theorem} 
By Merel's uniform boundedness theorem \cite{Merel96}, we now know that for a fixed degree $d$ there are only finitely many modular curves $X_1(N)$ with a non-cuspidal degree $d$ point. Thus a modern take on the classical classification problem is to determine the $N$ for which $X_1(N)$ admits a non-cuspidal degree $d$ point. This is known for $d \leq 4$ \cite{mazur77, kenkudegleq3, kamienny2degleq3, DEvHMZB2021, classificationdeg4points}. In degree 2, we again find that a non-cuspidal point exists only when it belongs to an infinite parameterized family: if it has positive genus, then $X_1(N)$ has a non-cuspidal quadratic point only when it arises from a degree 2 map to $\mathbb{P}^1$. However, for $d>2$, an essential challenge to classifying degree $d$ points on $X_1(N)$ is the need to control torsion structures that occur for only finitely many elliptic curves. Explicit examples of these exceptional points are known for $d=3$ and $5 \leq d \leq 13$, as discussed in \cite[Appendix A]{classificationdeg4points}, as well as many other suspected examples of higher degree.

In general, if a closed point does not belong to a infinite parameterized family of points of the same degree, we say it is \textbf{isolated}; see Section \ref{IsolatedPoints} for a more precise definition. The problem of classifying degree $d$ isolated points is the only obstruction to extending the classification of degree $d$ points on $X_1(N)$ to $5 \leq d \leq 9$ by \cite[Theorem 3]{DerickxVanHoeij2014} and \cite[Theorem 1.1]{NajmanVarivoda}. Though any curve has only finitely many isolated points, it can be quite difficult to determine them computationally, especially when the curve's Jacobian has positive rank over $\Q$. Indeed, if the genus of a curve is at least 2, then in particular all rational points are isolated, and this set can be notoriously difficult to compute. However, in the case of modular curves, we can use the moduli interpretation to pose new questions about these mysterious points.

One line of investigation first introduced by the author in collaboration with Ejder, Liu, Odumodu, and Viray \cite{BELOV} is the problem of classifying the elliptic curves which give rise to isolated points. We say $j\in X_1(1)$ is an \textbf{isolated $j$-invariant} if there exists an isolated point $x\in X_1(N)$ with $j(x)=j$ under the natural map $j:X_1(N) \rightarrow X_1(1)$.\footnote{Note $j$ is not itself an isolated point on $X_1(1) \cong \mathbb{P}^1$, which has no isolated points.} The \textbf{degree} of an isolated $j$-invariant is the degree of its residue field $\Q(j)$. For example, the modular curve $X_1(21)$ has 2 isolated points of degree 3 as discovered by Najman \cite{najman16}, each associated to an elliptic curve with $j$-invariant $-140625/8$. Thus $-140625/8$ is an isolated $j$-invariant of degree 1. Since the degree of the residue field of a closed point on $X_1(N)$ is determined by an elliptic curve defined over $\overline{\Q}$, it is enough to consider elliptic curves up to geometric isomorphism. 

 One collection of isolated $j$-invariants that is relatively well understood is the set of CM $j$-invariants, which correspond to elliptic curves with complex multiplication. Any CM $j$-invariant gives rise to isolated points on $X_1(N)$ for infinitely many integers $N$ by \cite[Theorem 7.1]{BELOV}, and conversely any $X_1(N)$ for $N\geq 721$ has an isolated CM point \cite[Theorem 8.2]{CGPS2022}. However, many non-CM isolated $j$-invariants are known. We summarize low degree examples in Table \ref{tab:non_cm_j_invariants}, where $r(k)$  indicates there are $k$ closed points of degree $r$. This is largely based on van Hoeij's table of low degree points \cite{vanHoeij}, and we provide full justification in Section $ \ref{ExampleSection}$. The 4 non-CM isolated $j$-invariants in $\Q$ are conjectured to be the only ones \cite{Algorithm2025,TeraoX1}.

\begin{table}[htbp]
    \centering
    \renewcommand{\arraystretch}{1.3}
    \begin{tabular}{>{\centering\arraybackslash}p{.3cm} >{\centering\arraybackslash}p{3.5cm} >{\centering\arraybackslash}p{6.7cm} >{\centering\arraybackslash}p{4.2cm}}
        \toprule
        \textbf{$d$} & \textbf{\# non-CM degree $d$ isolated $j$-invariants} & \textbf{Curves} & \textbf{Degrees of associated isolated points} \\
        \midrule
        1  & 4 & $X_1(21), X_1(28), X_1(37)$ & 3, 9, 6 \& 18 \\
        2  & 1 & $X_1(34)$  & 8 \\
        3  & 4 & $X_1(25), X_1(31), X_1(39), X_1(42)$ & 6, 9, 9, 9 \\
        4  & 6  &  $X_1(33), X_1(34),X_1(35), X_1(39), X_1(40), X_1(42)$ & 8, 8, 8, 12, 8, 8 \\
        5  & 5  & $X_1(28), X_1(29), X_1(30), X_1(40)$ & 5, 10, $5{(2)}$, 10 \\
       6  & 2 & $X_1(39)$ & 12(2) \\
        7  & 4  & $X_1(25),X_1(33),X_1(36)$ & 7, 7, 7(2) \\
        8  & 2  & $X_1(42)$ & 8(2) \\
        9  & 15 & $X_1(29), X_1(32), X_1(33)$, $X_1(34)$, $X_1(40)$ & 9, 9(8), 9, 9(4), 9 \\
       10 & 8 & $X_1(29),X_1(35), X_1(38)$, $X_1(39)$, $X_1(40)$  & 10, 10(2), 10(3), 10, 10\\
        \bottomrule
    \end{tabular}
    \caption{Known non-CM isolated $j$-invariants of degree $\leq 10$.} 
    \label{tab:non_cm_j_invariants}
\end{table}
\vspace{-.5cm}

As consequence of Merel's uniform boundedness theorem \cite{Merel96}, there are only finitely many non-cuspidal isolated points of fixed degree on $X_1(N)$ as $N$ ranges over all integers. However, this does not imply finiteness of isolated $j$-invariants of degree $d$. Thus we can pose a new uniformity problem.

\begin{question}[Bourdon, Ejder, Liu, Odumodu, Viray \cite{BELOV}]\label{MainQuestion}
Are there only finitely many isolated $j$-invariants of each fixed degree?
\end{question}

In this article, we will explain how Question \ref{MainQuestion} relates to other uniformity problems in the literature and provide new results in the case of elliptic curves with rational $j$-invariant.

\subsection{Relation to Other Uniformity Questions} There are a number of other hypotheses one can make concerning the uniformity of Galois representations of elliptic curves over number fields of a fixed degree. Recall $X_0(N)$ is the smooth projective curve over $\Q$ whose non-cuspidal points correspond to elliptic curves with a rational cyclic isogeny of degree $N$.
\\
\noindent 
\begin{center}
    \fbox{\begin{minipage}{.96\textwidth} 
\begin{hypothesis}
    For each $d \in \Z^+$, there exists $C=C(d)$ such that for all non-CM elliptic curves $E$ defined over number fields of degree $d$ and $p>C$, the $p$-adic Galois representation of $E$ contains $\SL_2(\Z_{p})$. 
\end{hypothesis}

\begin{hypothesis}
    For each $d \in \Z^+$, there are only finitely many isolated $j$-invariants of degree $d$.
\end{hypothesis}

\begin{hypothesis}
    For each $d \in \Z^+$, there exists an $C=C(d)$ such that for all $N>C$, the modular curve $X_0(N)$ has no non-cuspidal, non-CM points of degree $d$.
\end{hypothesis}

\begin{hypothesis}
    For each $d \in \Z^+$, there exists $C=C(d)$ such that for all non-CM elliptic curves $E$ defined over $F$ with $[\Q(j(E)):\Q]=d$ we have  $
    \exp E(F)_{\tors} \leq C \cdot [F:\Q]^{1/2}
    $ and $\#E(F)_{\tors} \leq C \cdot [F:\Q]$.
\end{hypothesis}
\end{minipage}}
\end{center}
\vspace{.3cm}

When $d=1$, these hypotheses are all established conjectures or theorems. In this case, Hypothesis 1 is often referred to as ``Serre's Uniformity Conjecture", though it was a question of Serre \cite[$\S4.3$]{serre72} which has since been conjectured by both Sutherland \cite{sutherland} and Zywina \cite{ZywinaImages}. Hypothesis 3 for $d=1$ is a theorem of Mazur \cite{mazurX0n} and Kenku \cite{kenkuX0n}, and we can take $C=37$. Variants of Hypothesis 4 originate with Clark and Pollack \cite[Conjecture 1.4]{CP18}, though the exponent is not specified. It refines various ``folklore" polynomial bounds conjectures appearing in the literature (e.g., \cite[Conjecture 1]{CCS13}).  Prior work of the author and Genao \cite{BourdonGenao} shows that when $d=1$, for any $\epsilon>0$ there exists a constant $C_{\epsilon}$ such that $\exp E(F)_{\tors} \leq C_{\epsilon} \cdot [F:\Q]^{1+\epsilon}
    $. 

    For $d>1$, less has been formally conjectured, aside from Hypothesis 4 in \cite[Conjecture 1.4]{CP18}, and Hypothesis 3 in the case where $d=2$ as in \cite[Conjecture 1.1]{QuadPointsConj} and \cite[Conjecture 18]{BalakrishnanMazur}. We note  that the statement of Hypothesis 1 was suggested as a possible generalization of Serre's Uniformity Conjecture by Derickx \cite{DerickxPost}, and Hypothesis 3 would imply an affirmative answer to a question of Mazur \cite[Question C]{MazurBonn}; see also Hypothesis SI($d$) in \cite{CP18}.  However, in this paper we establish general implications between these hypotheses, as illustrated in Figure \ref{ImplicationDiagram}. This is proven in Theorem \ref{implication_thm}, with the  implication from Hypothesis 1 to Hypothesis 2 coming from \cite{BELOV}. Thus we can view Question \ref{MainQuestion} as a refinement of Hypotheses 3 and 4 which lacks the full strength of generalized Serre uniformity.

    \begin{remark}      If $C$ is taken to be an absolute constant in Hypothesis 4, this would actually imply Hypothesis 1, as shown in $\S7$. This is related to a question of Hindry and Silverman \cite[$\S3$]{HS99}.
    
    We note that \cite[Theorem 1.7]{CP18} states  Hypothesis 3 implies a weaker version of Hypothesis 4. However, this proof relies on \cite[Theorem 1.2]{LR16}, which contains an error. See \cite[$\S 1$]{Smith23} for details. 
    \end{remark}

\begin{figure}
    \begin{center}
\begin{tikzpicture}[
    box/.style={draw, rectangle, minimum width=2.5cm, minimum height=1cm, align=center, thick},
    arrow/.style={-Latex, thick} 
]

\node[box] (H1) {Hypothesis 1: Generalized Serre Uniformity};
\node[box] (H2) [below=1cm of H1] {Hypothesis 2: Finiteness of Isolated $j$-Invariants};
\node[box] (H3) [below=1cm of H2, xshift=-4cm] {Hypothesis 3: Non-CM Isogeny Bounds};
\node[box] (H4) [below=1cm of H2, xshift=4cm] {Hypothesis 4: Refined Polynomial Bounds};

\draw[arrow] (H1) -- (H2);
\draw[arrow] (H2) -- (H3);
\draw[arrow] (H2) -- (H4);

\end{tikzpicture}
\end{center}
\caption{Implications Among Hypotheses (Theorem \ref{implication_thm})} \label{ImplicationDiagram}
\end{figure}
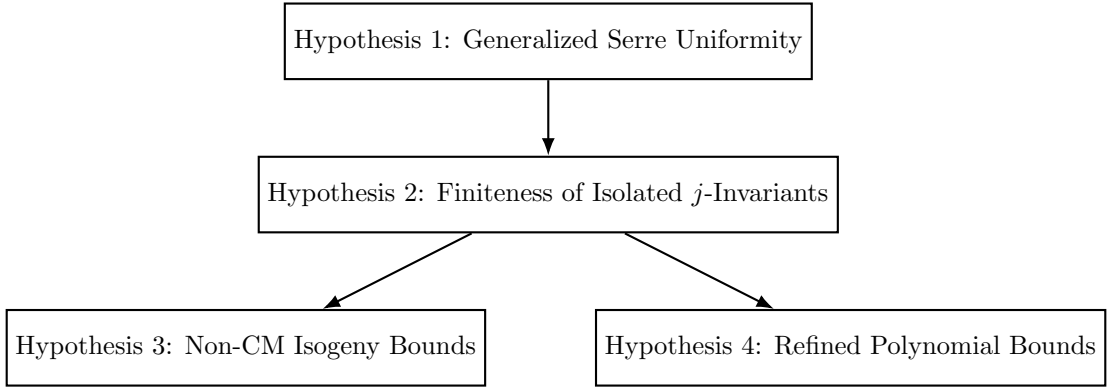

\subsection{Unconditional Progress for Rational $j$-invariants} Though Serre's Uniformity Conjecture is still open, we have partial classification results concerning the image of $p$-adic Galois representations associated to non-CM elliptic curves over $\Q$. As shown by Ejder \cite{Ejder22}, these are strong enough to imply there are no non-CM isolated points $x \in X_1(p^a)$ with $j(x) \in \Q$ and $p>37$. Thus extending unconditional results requires controlling possible ``entanglement" of torsion point fields associated to distinct primes. Of key importance are criteria which imply for $x\in X_1(N)$ with large prime $p \mid N$ that the residue field contribution at $p$ is nearly as large as it can be. We establish a new result of this form in Corollary \ref{IntroCor}, which improves \cite[Proposition 5]{BourdonGenao}. When combined with work of Lemos \cite{lemosTrans,lemosZ}, this is sufficient to obtain our next result.

\begin{theorem}\label{finiteness_2_primes}
    There are only finitely many rational isolated $j$-invariants associated to $X_1(p^aq^b)$ for primes $p,q$ and non-negative integers $a,b$.
\end{theorem}

The same ingredients bring us closer to establishing Hypothesis 4 for elliptic curves with rational $j$-invariants. Our final main result improves bounds of the author and Genao \cite[Theorem 2]{BourdonGenao} by a square root factor of the degree as well as earlier bounds of Clark and Pollack \cite[Theorem 1.3]{CP18}.

    \begin{theorem}\label{PolynomialBoundsThm}
        For any $\epsilon>0$, there exists a constant $C_{\epsilon}$ such that for all non-CM elliptic curves $E/F$ with $j(E) \in \Q$ we have  $
    \exp E(F)_{\tors} \leq C_{\epsilon} \cdot [F:\Q]^{1/2+\epsilon}
    $ and $\# E(F)_{\tors} \leq C_{\epsilon} \cdot [F:\Q]^{1+\epsilon}$.
    \end{theorem}
\noindent The exponent bound is nearly optimal, in that the $1/2$ cannot be replaced by anything smaller. This follows, for example, from the fact that any elliptic curve over $\Q$ attains a point of prime order $p$ in as extension of degree at most $p^2-1$. Moreover,
this is the first bound on the exponent  which is too small to permit the inclusion of CM elliptic curves. As shown in the proof of \cite[Theorem 6]{CCS13}, if $E/\Q$ is an elliptic curve with CM by $\Q(\sqrt{-3})$, then there exists a sequence of fields $F_n$ of degree $d_n$ such that $\exp(E(F_n)_{\tors})\gg d_n\sqrt{\log \log d_n}$. 

\subsection{Future Work} Though our control on entanglement is not sufficient to prove Hypothesis 2 unconditionally for rational $j$-invariants, experimental data suggests the problem may be accessible to certain formal immersion arguments. Indeed, a primary obstruction to proving Serre's Uniformity Conjecture is that the associated modular curves lack a nontrivial rank 0 quotient. However, the ``entanglement modular curves" relevant to strengthening Corollary \ref{IntroCor} often possess a such a quotient. See Remark \ref{Rank0Remark}.

\section*{Acknowledgments}
The author is grateful to Pete Clark, Tyler Genao, Sachi Hashimoto, Filip Najman, and Paul Pollack for providing many helpful comments on an earlier draft of this work. In particular, Clark's suggestion to use upper bounds on $\omega(N)$ significantly streamlined the original proof of Theorem \ref{PolynomialBoundsThm}. 

The author was partially supported by NSF grant DMS-2145270.

\section*{AI Statement}
Google Gemini 3.6 was used to assist with the creation of Table \ref{tab:non_cm_j_invariants}. In particular, prior work of van Hoeij \cite{vanHoeij} recorded the degree of $\Q(j(E))$, but these fields needed to be recreated to confirm isolated $j$-invariants were distinct. Google Gemini 3.6 was also used to obtain a reference for the bound on exceptional subgroups \cite[Remark 2.1]{GhateParent}, as applied in the proof of Theorem \ref{LastResult}. It created initial Latex code for Table \ref{tab:non_cm_j_invariants} and Figure \ref{ImplicationDiagram}. Finally, it provided feedback on the initial version of small sections of the article with suspected subtle grammar errors (namely, the abstract and the first paragraph of page 2). 

Claude Sonnet 4.6 was used to speed up preliminary data collection for the second paragraph of Remark \ref{Rank0Remark}. For specific primes $p$ and $q$, it was asked to return a list of modular curves in the  \href{https://beta.lmfdb.org/ModularCurve/Q/}{LMFDB database} covering the modular curve $C_{ns}^+(p) \times C_{ns}^+(q)$, along with their subgroup generators and Jacobian decomposition data.

The author takes full responsibility for the article and its mathematical correctness.

\section{Background}
\subsection{Galois Representations of Elliptic Curves} \label{Background}If an elliptic curve $E$ is defined over a number field $k$, then the absolute Galois group of $k$ acts naturally on $E(\overline{k})$. For points of finite order, this action is encoded in the \textbf{adelic Galois representation} associated to $E$. After choosing compatible bases, this can be represented as
\[
\rho_{E}: \Gal_k \rightarrow  \GL_2(\widehat{\Z}) \cong \prod_{p} \GL_2(\Z_p).
\]
For any positive integer $m$, we can compose $\rho_E$ with projection onto the $p$-adic factors for $p\mid m$, giving the the \textbf{$m$-adic representation},
\[
\rho_{E,m^{\infty}}: \Gal_k \rightarrow  \prod_{p\mid m} \GL_2(\Z_p).
\]
Alternatively, we can reduce $\rho_E$ modulo $m$ to obtain the \textbf{mod $m$ Galois representation}. This gives the Galois action on points of $E$ with order dividing $m$:
\[
\rho_{E,m}: \Gal_k \rightarrow   \GL_2(\Z/m\Z).
\]

Serre's Open Image Theorem \cite{serre72} implies that for a fixed non-CM elliptic curve $E/k$, there exists a constant $C=C(E)$ depending on $E$ such that the mod $p$ Galois representation is surjective for primes $p>C$.  It is interesting to determine the proper subgroups that can occur as mod $p$ images of Galois as $E$ ranges over all elliptic curves defined over $k$. If $k=\Q$ and $p \leq 17$, the groups that can occur as $\im \rho_{E,p}$ are known; this is work of Zywina \cite{ZywinaImages} for $p \leq 11$ and Balakrishnan, Dogra, M\"{u}ller, Tuitman, and Vonk \cite{Balakrishnan,17adic} for $p=13,17$. In this paper, we use the labels of Sutherland \cite{sutherland} to refer these subgroups of $\GL_2(\Z/p\Z)$ that arise as $\im \rho_{E,p}$. For larger primes, we have the following result. Here, $C_{ns}^+(p)$ denotes the normalizer of a non-split Cartan subgroup.
\begin{theorem}[Mazur \cite{mazurX0n}, Serre \cite{Serre1981}, Bilu, Parent, Rebolledo \cite{BPR13}, Furio, Lombardo \cite{FurioLombardo}] \label{SerreUnifProgress}
    Let $E/\Q$ be a non-CM elliptic curve with $j(E) \not\in\{ -9317, -162677523113838677\}$. If $p\geq 19$ is prime, then $\im \rho_{E,p}=\GL_2(\Z/p\Z)$ or $C_{ns}^+(p)$.
\end{theorem}
\noindent The exceptional $j$-invariants are associated to elliptic curves with a rational cyclic 37-isogeny. For any elliptic curve $E/\Q$ with $j(E) \in\{ -9317, -162677523113838677\}$ and $p>37$, the associated mod $p$ Galois representation is surjective.

Serre's Uniformity Conjecture states that there exists a single constant $C$ such that $\rho_{E,p}$ is surjective for all non-CM $E/\Q$ and primes $p>C$. This is still an open problem, with Theorem \ref{SerreUnifProgress} summarizing the  progress to date. By contrast, for any fixed positive integer $m$, there exist strong uniformity results for the level of the $m$-adic Galois representations of non-CM elliptic curves defined over number fields of a fixed degree. Here, the \textbf{level} of the $m$-adic Galois representation is defined to be the smallest positive integer $M$ such that $\im \rho_{E,m^{\infty}}=\pi^{-1}(\im \rho_{E,M})$, where $\pi: \GL_2(\Z_m) \rightarrow \GL_2(\Z/M\Z)$ is the natural reduction map.
\begin{theorem}[\cite{CT13,BELOV}]\label{prop:UniformLevelFiniteSetPrimes}
		Let $d$ be a positive integer and $\calE$ a set of non-CM elliptic curves over number fields of degree at most $d$.
Then for any positive integer $m$, there exists a  positive integer $M$ with $\Supp(M) \subset \Supp(m)$ such that for all $E/k \in \calE$ we have
            \[
				\im \rho_{E, m^{\infty}} = \pi^{-1}(\im \rho_{E, M}).
            \]
\end{theorem}

\begin{proof}
    If $m$ is a power of a single prime, this is \cite[Theorem 1.1]{CT13}. The proof of the general statement is given in \cite[Proposition 6.1]{BELOV}.
\end{proof}

In particular, for a fixed prime $p$, there are only finitely many groups that can occur as $\im \rho_{E,p^{\infty}}$ for all non-CM elliptic curves defined over number fields of a fixed degree. In the case of non-CM elliptic curves over $\Q$, the classification of $p$-adic images is complete for $p=2,3, 13$ and 17 by \cite{2adicImage,RouseSutherlandZB22,3adicImage}. Throughout, we use the labels introduced by Rouse, Sutherland, and Zureick-Brown \cite{RouseSutherlandZB22} to refer subgroups of $\GL_2(\Z_p)$ which can arise as $\im \rho_{E,p^{\infty}}$ for non-CM elliptic curves $E/\Q$. These labels are of the form \texttt{N.i.g.n}, where $N$ denotes the level, $i$ denotes group index in $\GL_2(\Z_p)$, $g$ denotes the genus of the associated modular curve, and $n$ distinguishes groups with the same values for $N,g,i$.

\subsection{Modular Curves} The modular curve $Y_1(N)$ is an affine curve over $\Q$ whose $k$-rational points correspond to elliptic curves $E/k$ with a point $P \in E(k)$ of order $N$. If $N \geq 4$, then $Y_1(N)(k)$ is in bijection with $k$-isomorphism classes of such pairs, $(E,P)_k$. Taking the projective closure of $Y_1(N)$ adds a finite number of cusps, which do not correspond to elliptic curves, resulting in the smooth projective curve $X_1(N)$.   Throughout we view $X_1(N)$ as a curve defined over $\Q$.

If $x\in X_1(N)$ is a closed point, then its \textbf{degree} is defined to be the degree of its residue field, denoted $\Q(x)$. It is also the length of the Galois orbit of the point in $X_1(N)(\overline{\Q})$ corresponding to $x$. Note that if $E/k$ is an elliptic curve with $P \in E(k)$ of order $N$, then the degree of the associated closed point $x=[E,P] \in X_1(N)$ may be less than the degree of $k$. However, there is always $E'$ with $j(E')=j(E)$ which is defined over $\Q(x)$ and has a $\Q(x)$-rational point of order $N$. For more details, see \cite[$\S3.2.3$]{Liu2002} and \cite[p. 274, Proposition VI.3.2]{DR}. Explicit examples illustrating the distinction between closed points, $k$-rational points, and geometric points of $X_1(N)$ can be found in \cite[$\S2.3$]{Algorithm2025}.

We define the modular curve $X_0(N)$ in an analogous way. There is an affine curve $Y_0(N)$ over $\Q$ whose $k$-rational points correspond to elliptic curves $E/k$ with a $k$-rational isogeny $\varphi:E \rightarrow E'$ which is cyclic of degree $N$. Since $k$-rational cyclic $N$-isogenies from $E/k$ are in bijection with $k$-rational cyclic subgroups of $E$ of order $N$, this provides an alternate moduli interpretation for the points of $Y_0(N)$. Taking the projective closure of $Y_0(N)$ gives the modular curve $X_0(N)$. Unlike $X_1(N)$ for $N \geq 4$, the curve $X_0(N)$ is not a fine moduli space.

We often make use of the following well-known degree formulas for maps between modular curves.

    \begin{proposition}\label{prop:degree}
                For positive integers $a$ and $b$, there is a natural $\Q$-rational map $f \colon X_1(ab) \rightarrow X_1(a)$ sending $[E,P]$ to $[E, bP]$ with
                \[
                    \deg(f)=
                    c_f\cdot b^2 \prod_{p \mid b, p \nmid a}\left( 1-\frac{1}{p^2} \right).             \]
                    Here, $c_f=1/2$ if $a \leq 2$ and $ab>2$ and $c_f=1$ otherwise. Similarly, there is a natural $\Q$-rational map $g:X_0(ab) \rightarrow X_0(a)$ sending $[E, \langle P \rangle]$ to $[E, \langle bP \rangle]$ with
        \begin{align*} \deg(g) =
            b \prod_{p \mid b, p \nmid a}\left(1 + \frac{1}{p}\right).
        \end{align*}
          \end{proposition}

          \begin{proof}
              The degree computations follow from \cite[p.66]{modular}.
          \end{proof}

\subsection{Isolated Points}\label{IsolatedPoints}
Let $C/k$ be a nice curve over a number field, i.e., that is smooth, projective, and geometrically integral. Assume for simplicity that $P_0\in C(k)$; for a more general setup, see \cite[$\S4$]{BELOV} or \cite{VirayVogt}. Then there is a natural map from from the $d$th symmetric product of $C$ to the curve's Jacobian,
\[
\Phi:\Sym^d(C) \rightarrow \Jac(C),
\]
defined by sending an effective degree $d$ divisor $D$ to the class $[D-dP_0]$. By identifying a degree $d$ closed point $x\in C$ with the sum of its associated Galois conjugates, we may view $x$ as an element of $\Sym^d(C)(k)$.

Let $x\in \Sym^d(C)(k)$ be a degree $d$ closed point. Then:
\begin{itemize}
    \item We say $x$ is \textbf{$\mathbb{P}^1$-parameterized} if $\Phi(x)=\Phi(y)$ for some $y\in \Sym^d(C)(k)$ with $y \neq x$. In this case, there exists a degree $d$ map $f: C \rightarrow \mathbb{P}^1$ with $x \in f^{-1}(\mathbb{P}^1(k))$.
    \item We say $x$ is \textbf{AV-parameterized} if there exists a positive rank abelian subvariety $A/k$ of $\Jac(C)$ with $\Phi(x)+A \subseteq \im(\Phi)$.
    \item If $x$ is neither $\mathbb{P}^1$-parameterized nor AV-parameterized, then it is \textbf{isolated}.
    \item If there are only finitely many points of degree at most $\deg(x)$, then $x$ is \textbf{sporadic}.
\end{itemize}

By the Riemann-Roch Theorem, any isolated point has degree at most the genus of $C$. Since degree $d$ points that are \emph{not} $\mathbb{P}^1$-parameterized map injectively into $\Jac(C)(k)$, it is a consequence of a theorem of Faltings \cite{faltings} that there are only finitely many isolated points of any fixed degree. Taken together, these imply $C$ has only finitely many isolated points of any degree. Moreover, any parameterized point gives an infinite family of points of the same degree. For $\mathbb{P}^1$-parameterized points, this is Hilbert's Irreducibility Theorem \cite[Chapter 9]{serre97}. For AV-parameterized which are $\mathbb{P}^1$-isolated, this is proven in \cite[Theorem 4.2]{BELOV}, but see \cite[Proposition 4.3.5]{VirayVogt} for a more self-contained reference. We summarize this in the theorem below.
\begin{theorem}[Faltings \cite{faltings}, Bourdon, Ejder, Liu, Odumodu, Viray \cite{BELOV}]\label{IsolThm}
\leavevmode
    \begin{enumerate}
        \item There are only finitely many isolated points on $C$.
        \item There are infinitely many degree $d$ points if and only if there is a degree $d$ parameterized point.
    \end{enumerate}
\end{theorem}
\noindent In particular, every sporadic point is isolated, but the converse need not hold. 
\begin{figure}[h]

\hspace*{2.5cm}  \includegraphics[width=.8\linewidth]{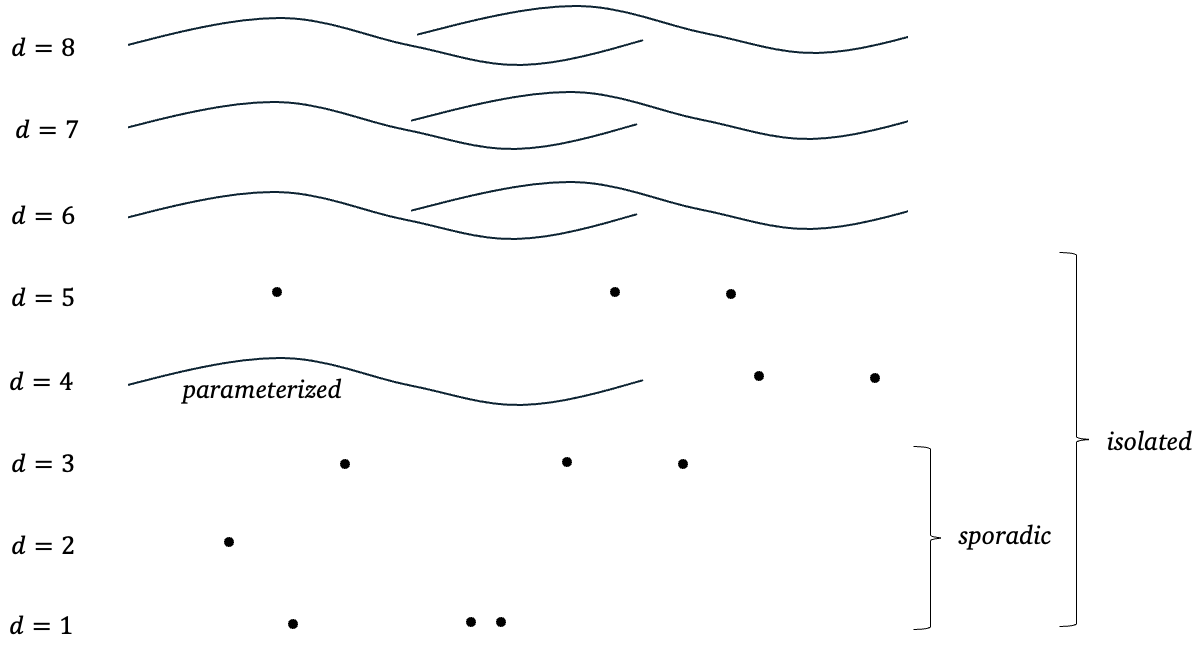} 
\caption{A hypothetical curve of genus 7 with points of degree $d$.}
\end{figure}

We will often make use of the following result which allows one to show the image of an isolated point is isolated under certain conditions.

\begin{theorem}[Bourdon, Ejder, Liu, Odumodu, Viray \cite{BELOV}]\label{FiberThm}
    Let $f:C \rightarrow D$ be a finite map of curves, and let $x\in C$ be an isolated point. If
    \[
    \deg(x)=\deg(f) \cdot \deg(f(x)),
    \]
    then $f(x)\in D$ is isolated.
\end{theorem}

\section{Implications Among Hypotheses}\label{Section3}
In this section, we prove the following implications among Hypotheses 1-4, as summarized in Figure \ref{ImplicationDiagram}.
\begin{theorem}\label{implication_thm} Let Hypotheses 1-4 as in $\S1.1$. Then:
   \begin{itemize}
       \item\cite[Corollary 1.7]{BELOV} Hypothesis 1 implies Hypothesis 2.
       \item Hypothesis 2 implies Hypothesis 3 and Hypothesis 4.
   \end{itemize}

\end{theorem}

We note that \cite[Corollary 1.7]{BELOV} as stated relies on a stronger conjecture than Hypothesis 1. As pointed out by Derickx \cite{DerickxPost}, the statement of \cite[Corollary 1.4]{BELOV} actually fails to hold over general number fields. However, the proof that Hypothesis 1 implies Hypothesis 2 still follows directly from the arguments found in \cite{BELOV}. Indeed, suppose Hypothesis 1 holds for $d \in \mathbb{Z}^+$. Let $m$ be the product of 2, 3, 5, and all primes $p$ with $p \leq C(d)$, and fix $M$ as in Theorem \ref{prop:UniformLevelFiniteSetPrimes} for $\mathcal{E}$ the set of all non-CM elliptic curves over number fields of degree $d$. Let $x\in X_1(N)$ be a non-cuspidal isolated point with $j(x)$ of degree $d$. Since there are only finitely many CM $j$-invariants of degree $d$, we may assume $j(x)$ is non-CM. Let $E/\Q(j(E))$ be an elliptic curve with $j(x)=j(E)$. Then by \cite[Theorem 5.1]{BELOV}, we have
\[
\deg(x)=\deg(f)\cdot \deg(f(x)),
\]
where $f:X_1(N)\rightarrow X_1(\gcd(N,M))$ is the natural map. By Theorem \ref{FiberThm}, the point $f(x)$ is isolated, and each modular curve $X_1(a)$ for $a \mid M$ has only finitely many isolated points by Theorem \ref{IsolThm}. The result follows.

It remains to prove Hypothesis 2 implies Hypotheses 3 and 4, which we do in $\S3.1$ and $\S3.3$, respectively.

\subsection{Implication for Hypothesis 3} Here we prove that Hypothesis 2 implies Hypothesis 3.
\begin{proposition}\label{Prop_Hyp2_3}
    Suppose there are only finitely many isolated $j$-invariants of degree $d$ for each $d\in \Z^+$. Then for each $d$ there exists a $C=C(d)$ such that for all $N>C$, the modular curve $X_0(N)$ has no non-cuspidal, non-CM points of degree $d$.
\end{proposition}

\begin{proof}
    Let $x\in X_0(N)$ be a non-cuspidal, non-CM point of degree $d$. If $N \geq 3$, then $x$ lifts to a point $x'\in X_1(N)$ with $\deg(x') \leq d \cdot \varphi(N)/2$, where $\varphi$ denotes the Euler $\varphi$-function. For sufficiently large $N$, lower bounds on gonality due to Abramovich \cite[Theorem 0.1]{abramovich} imply $\deg(x')< \frac{1}{2}\gon_{\Q}(X_1(N))$. Thus $x'$ is sporadic --- and hence isolated -- by work of Frey \cite[Proposition 2]{frey}. By assumption, there are only finitely many isolated $j$-invariants of each degree and $\deg(j(x')) \leq d$, so $j(x')$ belongs to a finite set. By Serre's Open Image Theorem \cite{serre72}, each $j(x')$ gives a degree $d$ point on $X_0(N)$ for at most finitely many integers $N$.
\end{proof}

\begin{remark}
In fact, it is enough to assume Hypothesis 2 for modular curves of prime level.
Indeed, suppose there are only finitely many isolated $j$-invariants of each fixed degree associated to modular curves $X_1(p)$ for $p$ prime. Then the proof of Proposition \ref{Prop_Hyp2_3} shows the modular curve $X_0(p)$ has no non-cuspidal, non-CM points of degree $d$ for sufficiently large primes $p$.  Hypothesis 3 now follows from Theorem \ref{prop:UniformLevelFiniteSetPrimes}.
\end{remark}

\subsection{Preliminary Lower Bounds} We record the following technical lemma in more generality than is needed here so it can be applied later in the proof of Theorem \ref{PolynomialBoundsThm}.
\begin{lemma}\label{Finite_Degree_Bounds}
Fix $d \in \Z^+$, and let $\mathcal{E}$ be a collection of non-CM elliptic curves such that each $E\in \mathcal{E}$ is defined over $\Q(j(E))$ with $[\Q(j(E)):\Q]=d$. Let $S$ be a finite list of primes. Then there exists a constant $C=C(d)$ such that for any $x\in X_1(N)$ associated to $E \in \mathcal{E}$ with
\[
\Supp(N) \subseteq S \cup \{p: \im \rho_{E, p^{\infty}}=\GL_2(\Z_{p}) \}
\] one has
   \[
    \deg(x)\geq \frac{d}{C} \cdot \deg(X_1(N) \rightarrow X_1(1)).
    \]
\end{lemma}
\begin{proof}
    We enlarge $S$ to include 2 and 3, if necessary. Let $\mathfrak{m}=\prod_{p \in S} p$. Then by Theorem \ref{prop:UniformLevelFiniteSetPrimes}, there exists a positive integer $M$ with $\Supp(M) \subseteq S$ such that for each $E_i/\Q(j(E_i))$ we have
    \[
    \im \rho_{E_i,\mathfrak{m}^{\infty}}=\pi^{-1}(\im \rho_{E_i,M}),
    \]
    where $\pi$ is the natural projection map.

    Let $N\in \Z^+$, and $x \in X_1(N)$ associated to $E_i$. Suppose \[\Supp(N) \subseteq S \cup \{p: \im \rho_{E_i, p^{\infty}}=\GL_2(\Z_{p})\}.\] Write $N=ab$, where $\Supp(a) \subseteq S$ and $\Supp(b) \cap S = \varnothing$. Since $b$ is the product of primes for which the $p$-adic Galois representation associated to $E_i$ is surjective, by \cite[Proposition 5.7]{BELOV} we have
    \[
\deg(x)=\deg(f)\cdot\deg(f(x)),
    \]
    where $f:X_1(N) \rightarrow X_1(a)$ is the natural map. By \cite[Proposition 5.8]{BELOV}, we have
    \[
    \deg(f(x))=\deg(g) \cdot \deg(g(f(x))),
    \]
    where $g: X_1(a) \rightarrow X_1(\gcd(a,M))$ is the natural map. Since $\deg(g(f(x))) \geq d$, the claim follows with
    \[
    C=\deg(X_1(M) \rightarrow X_1(1)). \qedhere
    \]
\end{proof}
\begin{proposition}\label{Prop_Isol_Lower_Bound}
    Let $d\in \Z^+$. Suppose there are only finitely many isolated $j$-invariants of degree $d$. Then there exists a constant $C=C(d)$ such that for any non-CM elliptic curve $E$ with $[\Q(j(E)):\Q]=d$ and $x\in X_1(N)$ with $j(x)=j(E)$ we have
    \[
    \deg(x) > C \cdot N^2.
    \]
\end{proposition}

\begin{proof}
    Let $j_1, j_2, \dots, j_k$ be the non-CM isolated $j$-invariants of degree $d$. For each $1 \leq i \leq k$, fix $E_i/\Q(j(E_i))$ with $j(E_i)=j_i$. By Serre's Open Image Theorem \cite{serre72}, there exists a prime $p_0>3$ such that the $p$-adic Galois representation of $E_i/\Q(j(E_i))$ is surjective for $p>p_0$. Taking $S=\{\text{primes} \leq p_0\}$ in Lemma \ref{Finite_Degree_Bounds}, there exists a constant $C_1=C_1(d)$ such that for any $N \in \mathbb{Z}^+$ and any $x=[E_i,P_i] \in X_1(N)$ we have 
    \begin{align*}
    \deg(x)&\geq \frac{d}{C_1} \cdot \deg(X_1(N) \rightarrow X_1(1))\\ &\geq\frac{d}{2 C_1} \cdot N^2 \prod_{p \mid N}\left( 1- \frac{1}{p^2} \right)\\
    & >\frac{d}{2 C_1} \cdot N^2 \prod_{p \text{ prime}}\left( 1- \frac{1}{p^2} \right) \\
    &=\frac{d}{2C_1}\cdot \frac{1}{\zeta(2)}\cdot N^2.
    \end{align*}
    Here $\zeta$ denotes the Riemann zeta function. Our lower bound argument is modeled after \cite[Proposition 5]{DedekindRef}.
    
    So now assume $j$ is not isolated, with $[\Q(j):\Q]=d$. Then in particular there is no sporadic point on $X_1(N)$ with $j(x)=j$. Thus by lower bounds on gonality due to Abramovich \cite{abramovich}, any point on $X_1(N)$ has degree at least \[\frac{1}{2}\gon_{\Q}(X_1(N)) \geq \frac{7}{1600}[\text{PSL}_2(\Z):\Gamma_1(N)]= \frac{7}{3200} N^2\prod_{p \mid N} \left(1-\frac{1}{p^2}\right) \cdot  >\frac{7}{3200}\cdot \frac{1}{\zeta(2) }\cdot N^2.\] 
    
   The statement holds with $C=\min\left\{\frac{d}{2C_1\cdot \zeta(2)},\frac{7}{3200\cdot \zeta(2)}\right\}$.
\end{proof}
\subsection{Implication for Hypothesis 4}
    Let $d\in \Z^+$. Suppose there are only finitely many isolated $j$-invariants of degree $d$. Suppose $E/F$ is a non-CM elliptic curve with $[\Q(j(E)):\Q]=d$ and $E(F)$ contains a point $P$ of order $N$. Then $\Q(x) \hookrightarrow F$, where $x=[E,P] \in X_1(N)$. By Proposition \ref{Prop_Isol_Lower_Bound}, there is a constant $C=C(d)$ depending on $d$ such that
    \[C\cdot N^2 < \deg(x) \leq [F:\Q].\]
 Since $\#E(F)_{\tors} \mid (\exp E(F)_{\tors})^2$, the bound on the full torsion subgroup follows. Taking the square root of both sides gives the desired bound on $N$.

\section{Rational Isolated $j$-invariants for $X_1(p^aq^b)$}\label{Section4}

    Here, we prove Theorem \ref{finiteness_2_primes}, which states that there are only finitely many rational isolated $j$-invariants associated to $X_1(p^aq^b)$, where $p,q$ are primes. Section \ref{prelim_Lemos} includes preliminary results which rely on work of Lemos \cite{lemosTrans,lemosZ}, whereas Section \ref{RamificationResults} identifies the contribution from ramification results of Smith \cite{Smith23}. These are applied to prove  Theorem \ref{finiteness_2_primes} in Section \ref{Finiteness_proof}.
\subsection{Consequences of Work of Lemos}\label{prelim_Lemos} 
\begin{lemma}\label{Lemma1}
Let $E/\Q$ be a non-CM elliptic curve, and let $p<q$ be primes with $q>37$. If $\im\rho_{E,q}=C_{ns}^+(q)$, then one of the following holds:
\begin{enumerate}
    \item There is a single closed point point on $X_1(p^a)$ associated to $E$, and it is of maximal degree.
    \item $p=2$ and the 2-adic image has RSZB label 4.8.0.2, 4.16.0.2, or 8.16.0.3.
\end{enumerate} 
\end{lemma}
\begin{proof}
    By work of Lemos \cite{lemosTrans,lemosZ}, the mod $p$ image is not contained in a Borel subgroup, or in the normalizer of a split Cartan subgroup. Thus by classification results in Section \ref{Background}, we can assume the mod $p$ Galois representation is surjective, has image equal to $C_{ns}^+(p)$, or is one of the following exceptional cases: $p=2$ and the image is 2Cn or $p=5,13$ and the image is $p$S4. We consider each case separately.
    \begin{itemize}
            \item If the mod $p$ image is surjective and $p \geq 5$, the result follows (e.g., \cite[Corollary 2.13]{Greicius10}).
        \item Suppose $p=3$ and mod 3 image is surjective. Then by the classification of 3-adic images \cite{RouseSutherlandZB22,3adicImage}, we know $\im \rho_{E,3^{\infty}}=\GL_2(\Z_3)$ or  9.27.0.1. In either case, we apply \cite[Proposition 34]{Algorithm2025} to compute the degree of a closed point from this $3$-adic image data to confirm that the result follows.
        \item Suppose $p=2$ and the mod 2 image is surjective. The classification of 2-adic images \cite{2adicImage} shows the claim holds unless the 2-adic image is 4.8.0.2, 4.16.0.2, or 8.16.0.3, as disired.
        \item Suppose the mod $p$ image is equal to $C_{ns}^+(p)$. We may assume $p$ is odd, and let $G$ be the $p$-adic image of Galois. By \cite[Theorem 6.5]{Furio2024} and \cite[Proposition 2.2]{Furio2024}, we are in one of the following cases:
        \begin{enumerate}
            \item $G$ has level $p^n$ and $G\pmod{p^n}=C_{ns}^+(p^n)$. Then the first bullet point of the proof of \cite[Theorem 6]{BourdonEjder} shows there is a single closed point on $X_1(p^n)$ associated to $E$ of maximal degree, and the claim now follows from \cite[Theorem 5.1]{BELOV}.
            \item $G$ has level $p^2$ and

            \[ G \pmod{p^2} \cong C_{\text{ns}}^+(p) \ltimes \left\lbrace I + p\begin{bmatrix}
a& \epsilon b\\
-b & c
\end{bmatrix}\right \rbrace \] 
where $\epsilon$ is a fixed integer which is not a quadratic residue modulo $p$. 
            Let $F=\Q(E[p])$. In this case, there is a choice of basis $\{P,Q\}$ for $E[p^2]$ such that the image of $\rho_{E/F,p^2}$ is $\left\{I + p \begin{pmatrix} a & \epsilon b \\ -b & c
            \end{pmatrix}\right\} \leq \GL_2(\Z/p^2\Z)$. An element of this form fixes the first basis element if and only if it is one of the $p$ matrices of the form
            \[
            \begin{pmatrix} 1 & 0 \\ 0 & 1+pc
            \end{pmatrix}.
            \] Thus $F(P)/F$ has degree $p^2$, which means $\Q(P)/\Q(pP)$ has degree $p^2$. Since $[\Q(pP):\Q]$ has degree $p^2-1$, it follows that $\Q(P)$ has largest possible degree and all elements of order $p^2$ are in a single Galois orbit. The claim now follows from \cite[Theorem 5.1]{BELOV}.
\end{enumerate}
        \item Suppose $p=2$ and the image is 2Cn. By \cite{2adicImage}, it follows that $\im \rho_{E,2^{\infty}}$ is 2.2.0.1, 4.4.0.2, or 8.4.0.1. In each case, there is a single closed point on $X_1(2^a)$ of maximal degree.
        \item If $p=5$ and the image is 5S4, then \cite{RouseSutherlandZB22} shows the 5-adic image is 5.5.0.1. The claim follows. 
        \item If $p=13$ and the image is 13S4, then $j(E)$ is one of 3 known exceptional $j$-invariants by \cite[$\S5.1$]{17adic}. A computation shows that the mod $q$ Galois image is surjective, and so in fact this case does not occur. See for example \cite[Lemma 17]{MayleIsogenyPaper}.\qedhere
    \end{itemize}
\end{proof}

\subsection{Consequence of Work of Smith}\label{RamificationResults}  We model the proof of \cite[Theorem 7.3]{Furio2024}, which uses ramification results of Smith \cite{Smith23}. This improvement upon \cite[Proposition 5]{BourdonGenao} is crucial for the proof of Theorem \ref{finiteness_2_primes}.
\begin{lemma} \label{Lemma2}
    Let $E/\Q$ be a non-CM elliptic curve, and let $q>17$ be prime with $\im\rho_{E,q}=C_{ns}^+(q)$. Let $F=\Q(E[N])$ for some $N$ with $\gcd(N,q)=1$. Then if $P$ is a point of order $q^b$, we have 
        \[
        \frac{q^{2b}-q^{2b-2}}{d} \leq [F(P):F],
        \]
        where $d \in \{1,2,3,4,6\}$. 
\end{lemma}
\begin{proof}
    By \cite[Proposition 3.1]{Ejder22} and \cite[Theorem 4.5]{Furio2024}, the curve $E$ has potentially good supersingular reduction at $q$ and no canonical subgroup of order $q$. Let $L/\Q_q^{nr}$ be the minimal extension over which $E$ attains good reduction. Since $q>17$, we know $d=[L:\Q_q^{nr}] \in \{1,2,3,4,6\}$; see \cite[$\S5.6$]{serre72}. By the N\'{e}ron-Ogg-Shafarevich criterion \cite[Theorem 1]{SerreTate}, the extension $L(E[N])/L$ is unramified, so $L(E[N])=L$. In particular $\Q_q(E[N]) \subseteq L$. Work of Smith \cite[Theorem 4.6]{Smith23} implies $L(P)$ contains an element with valuation $1/(q^{2b}-q^{2b-2})$, and so 
    \[
    [L(P):L] \geq \frac{q^{2b}-q^{2b-2}}{d}.
    \]
    Since $\Q_q(E[N]) \subseteq L$, we may conclude
    \[
    [\Q_q(E[N])(P):\Q_q(E[N])] \geq \frac{q^{2b}-q^{2b-2}}{d}. 
    \]
    The result follows.
\end{proof}

\begin{corollary}\label{IntroCor}
Let $x\in X_1(N)$ be a non-CM point with $j(x)\in\Q$. Suppose $q \mid N$ for some prime $q > 37$, and set $\ord_{q}(N)=b$. Then \[
\deg(x)\geq\frac{1}{6}\deg(f(x))\deg(f),
\] where $f:X_1(N) \rightarrow X_1(N/q^b)$ is the natural map.    
\end{corollary}

\begin{proof}
    Let $E/\Q$ be a non-CM elliptic curve with $j(E)=j(x)$. Since the degrees of closed points associated to $E$ are not effected by quadratic twist, we may assume $-I \in \im \rho_{E,N}$. By Theorem \ref{SerreUnifProgress}, the image of $\rho_{E,q}$ is $\GL_2(\Z/q\Z)$ or $C_{ns}^+(q)$. In the first case, the result follows from \cite[Proposition 5.7]{BELOV}. The second case is a consequence of Lemma \ref{Lemma2}. Indeed, let $x=[E,P]\in X_1(N)$ and $f(x)=[E,P']\in X_1(N/q^b)$. By Lemma \ref{Lemma2}, we have $[\Q(P):\Q(P')] \geq \frac{1}{6}\deg(f)$. If $N\not\in\{q^b,2q^b\}$, the result is now a consequence of \cite[Proposition 34]{Algorithm2025}. If $N\in\{q^b,2q^b\}$, this follows from \cite[Proposition 4]{BourdonEjder} and Lemma \ref{Lemma1}.
\end{proof}

\subsection{Proof of Theorem \ref{finiteness_2_primes}} \label{Finiteness_proof}
    Let $x=[E,P] \in X_1(p^aq^b)$ be isolated with $j(x) \in \Q$, and fix an equation for $E/\Q$. Since there are only 13 CM $j$-invariants in $\Q$, we may assume $j(x)$ is non-CM, and by \cite{BourdonEjder} we may assume $a,b>0$. 
    \begin{itemize}
        \item Suppose $p,q\leq 37$. Let $S=\{ \text{primes} \leq 37\}$, and let $m = \prod_{p \in S} p$. Then by Theorem \ref{prop:UniformLevelFiniteSetPrimes}, there is an absolute bound $M$ on the level of the $m$-adic Galois representation for all non-CM elliptic curves over $\Q$. By \cite[Proposition 5.8]{BELOV}, 
        \[
        \deg(x)=\deg(f)\deg(f(x)),
        \]
        where $f:X_1(p^aq^b)\rightarrow X_1(\gcd(p^aq^b,M))$ is the natural map. By Theorem \ref{FiberThm}, we see $f(x)$ is an isolated point on one of finitely many target curves. Each of these has only finitely many isolated points by Theorem \ref{IsolThm}, and hence there are only finitely many isolated $j$-invariants from this case.
        \item Assume $q> 37$ with $\im \rho_{E,q}=\GL_2(\Z/q\Z)$. Then by \cite[Proposition 5.7]{BELOV}, 
        \[
        \deg(x)=\deg(f) \cdot \deg(f(x)),
        \]
        where $f=X_1(p^aq^b) \rightarrow X_1(p^a)$. Thus it follows from Theorem \ref{FiberThm} that $f(x) \in X_1(p^a)$ is isolated. Thus $j(x)$ is on the finite list of isolated $j$-invariants coming from \cite{BourdonEjder}.
        \item Assume $q>37$ with $\im \rho_{E,q}\neq\GL_2(\Z/q\Z)$. Then by Theorem \ref{SerreUnifProgress}, we have  $\im \rho_{E,q}=C_{ns}^+(q)$. Suppose first that $p \neq 2$, or the 2-adic image is not 4.8.0.2, 4.16.0.2, or 8.16.0.3. By Lemma \ref{Lemma1}, we may assume the associated point on $X_1(p^a)$ is of maximal degree. Then by Corollary \ref{IntroCor},
\[
       \deg(x) \geq \frac{p^{2a-2}(p^2-1)\cdot q^{2b-2}(q^2-1)}{12}.
        \]
        This is strictly larger than the genus of $X_1(p^aq^b)$ by  \cite[Proposition 1.40]{Shimura71}, so the point is not isolated. 

        So suppose $p=2$ and the 2-adic image is 4.8.0.2, 4.16.0.2, or 8.16.0.3. Thus this means the point on $X_1(2^a)$ has degree $2^{\max(2a-4,0)}(2^2-1)$.
        As before, an application of Corollary \ref{IntroCor} shows 
\[
        \deg(x) \geq \frac{2^{2a-2}(2^2-1)\cdot q^{2b-2}(q^2-1)}{24},
        \]
        which is strictly larger than the genus by \cite[Proposition 1.40]{Shimura71}. Thus $x$ is not isolated.  
    \end{itemize}

\begin{remark}\label{Rank0Remark} An analogous argument cannot be used for products of more than 2 distinct primes, as the bounds of Corollary \ref{IntroCor} are not sufficient to prove the degree of $x\in X_1(N)$ is larger than the genus. Thus to obtain more general results, one must show the factor of $1/6$ appearing in Corollary \ref{IntroCor} can be improved. By the proof, it suffices to consider the case where the associated elliptic curve $E/\Q$ has $\im \rho_{E,q}=C_{ns}^+(q)$. If $\im \rho_{E,N} \cong \im \rho_{E,N/q^b} \times \im \rho_{E,q^b}$, then the factor $1/6$ is not needed, so this motivates the study of certain ``entanglement" modular curves.

When $\Q(E[p^a]) \cap \Q(E[q^b])\neq \Q$ for distinct primes $p$ and $q$, then we say there is (horizontal) Galois entanglement. This means that $\im \rho_{E,p^aq^b}$ is properly contained in $\im \rho_{E,p^a} \times \im \rho_{E,q^b}$. Several recent papers propose a systematic study of entanglement; see for example \cite{DanielsMorrow22,DanielsLRMorrow23}. However, for applications to Hypothesis 2, we are only interested in entanglement which impacts the degrees of points on modular curves. For example, if $E/\Q$ has $\im \rho_{E,q}=C_{ns}^+(q)$ for $q>37$, then work of Lemos \cite{lemosTrans, lemosZ} shows the other non-surjective prime divisors $p \mid N$ typically have $\im \rho_{E,p}=C_{ns}^+(p)$. We have identified modular curves currently found in the \href{https://beta.lmfdb.org/ModularCurve/Q/}{LMFDB database} which correspond to possible mod $pq$ images with $\im \rho_{E,p}=C_{ns}^+(p)$, $\im \rho_{E,q}=C_{ns}^+(q)$, and for which the factor of $1/6$ in Corollary \ref{IntroCor} cannot be removed. For example: \href{https://beta.lmfdb.org/ModularCurve/Q/15.60.3.f.1/}{15.60.3.f.1}, \href{https://beta.lmfdb.org/ModularCurve/Q/21.126.4.a.1/}{21.126.4.a.1}, \href{https://beta.lmfdb.org/ModularCurve/Q/33.330.17.a.1/}{33.330.17.a.1}, \href{https://beta.lmfdb.org/ModularCurve/Q/35.420.29.f.1/}{35.420.29.f.1}, \href{https://beta.lmfdb.org/ModularCurve/Q/55.1100.81.d.1/}{55.1100.81.d.1}. 
In each case, the modular curve possesses a nontrivial rank 0 quotient. This is interesting since the modular curve associated with the fiber product $C_{ns}^+(p) \times C_{ns}^+(q)$ does \emph{not} possess such a quotient. This leaves open the possibility that one could prove there are only finitely many isolated $j$-invariants in $\Q$ using formal immersion arguments, as in other uniformity results.
\end{remark}
\section{Improved Polynomial Bounds for Rational $j$-invariant}
This section proves Theorem \ref{PolynomialBoundsThm}.
First, we show how results from Sections \ref{Section3} and \ref{Section4} can be applied to obtain preliminary lower bounds on the degree of a point on $X_1(N)$ associated to a non-CM elliptic curve with rational $j$-invariant. The proof of the main result follows in Section \ref{Section5.2}.

\subsection{Preliminary Bounds} This sharpens  \cite[Proposition 5]{BourdonGenao} for large non-surjective prime divisors.
\begin{proposition}\label{LowerBoundProp}
   There exists an absolute constant $C$ such that for any non-CM elliptic curve $E/\Q$ and any $x=[E,P]\in X_1(N)$ we have
    \[
    \deg(x) \geq C \cdot N^2\cdot \prod_{p \mid N}\frac{1}{6}\left(1-\frac{1}{p^2}\right) \]
\end{proposition}

\begin{proof}
   Let $E/\Q$ be a non-CM elliptic curve. Replacing $E$ by a twist if necessary, we may assume $-I \in \im \rho_{E,N}$. Set $S \coloneqq \{\text{primes} \leq 37\}$. Let $N=n_1\cdot n_2$, where $\Supp(n_1) \subseteq S$ and $\Supp(n_2) \cap S=\varnothing$. By repeatedly apply Corollary \ref{IntroCor}, we see 
        \[
        [\Q(P):\Q(n_2P)]\geq n_2^2\prod_{p \mid n_2}\frac{1}{6}\left(1-\frac{1}{p^2}\right).     \]
        By Lemma \ref{Finite_Degree_Bounds}, there is a constant $C_0$ which does not depend on $E$ such that
 \[
        [\Q(n_2P):\Q] \geq \frac{1}{C_0} \deg(X_1(n_1) \rightarrow X_1(1)) \geq \frac{1}{C_0} \cdot\frac{1}{2} \cdot n_1^2\cdot \prod_{p \mid n_1}\frac{1}{6}\left(1-\frac{1}{p^2}\right).
        \]
        Since $\deg(x)$ is at least $\frac{1}{2}\cdot [\Q(P):\Q] =\frac{1}{2}\cdot[\Q(P):\Q(n_2P)]\cdot[\Q(n_2P):\Q]$, the result follows. \qedhere
  
\end{proof}
\subsection{Proof of Theorem \ref{PolynomialBoundsThm}}\label{Section5.2}
Suppose $E/F$ is a non-CM elliptic curve with $j(E) \in \Q$, and let $P \in E(F)$ of order $N$.
    Let $d=[F:\Q]$. By Proposition \ref{LowerBoundProp} we have
    \begin{align*}
    d &\geq C\cdot N^2 \cdot  \prod_{p \mid N}\frac{1}{6} \cdot \left(1-\frac{1}{p^2}\right)\\
    & > C\cdot N^2 \cdot  \prod_{p \text{ prime}}\left(1-\frac{1}{p^2}\right)\cdot \prod_{p \mid N}\frac{1}{6}\\
    &=\frac{C}{\zeta(2)}\cdot N^2 \left(\frac{1}{6}\right)^{\omega(N)},
    \end{align*}
    where $\zeta$ denotes the Riemann zeta function and $\omega(N)$ counts the number of primes dividing $N$. 

    Suppose $N \geq 3$ and fix $0 < \epsilon <\frac{1}{4}$. We proceed as suggested by Pete Clark. By \cite[Theorem 11]{Robin83} we have $\omega(N)\leq 1.3841 \log(N)/\log(\log(N))$, and
    it follows that
    \[
    6^{\omega(N)}\leq 6^{c_1 \log_6(N)/\log(\log(N))}=N^{c_1/\log(\log(N))}.
    \]
    Thus for $N$ sufficiently large, we have $6^{\omega(N)}<N^{\epsilon}$, and so there exists a constant $c_{\epsilon,1}$ depending on $\epsilon$ such that $6^{\omega(N)}\leq c_{\epsilon,1} N^{\epsilon}$. Combining this with our initial inequality gives 
    \[
    d > \frac{C}{\zeta(2)}\cdot \frac{1}{c_{\epsilon,1}} \cdot N^{2-\epsilon}.
    \]
    Solving for $N$ shows 
     $
     N < c_{\epsilon,2} \cdot d^{1/(2-\epsilon)}<c_{\epsilon,2} \cdot d^{1/2 +\epsilon},
    $ as desired. Since $\#E(F)_{\tors} \mid (\exp E(F)_{\tors})^2$, the bound on the full torsion subgroup follows.

\section{Examples of Isolated $j$-invariants} \label{ExampleSection}

Here, we justify the examples of non-CM isolated $j$-invariants appearing in Table \ref{tab:non_cm_j_invariants}. Each point appears in van Hoeij's table of low degree points \cite{vanHoeij}. The degree 1 isolated $j$-invariants are described in detail in \cite{Algorithm2025}, so it suffices to provide justification for isolated $j$-invariants of degree $2 \leq d \leq 10$. 

\subsection{Justifying Isolated} If a curve's Jacobian has rank 0 over $\Q$, then any point in degree less than the $\Q$-gonality is isolated. The $\Q$-gonality of $X_1(N)$ is known for $1 \leq N \leq 40$ by work of Derickx and van Hoeij \cite{DerickxVanHoeij2014}, and the gonality of $X_1(42)$ is at least 10 by recent work of Najman and Varivoda \cite{NajmanVarivoda}. For isolated $j$-invariants in Table \ref{tab:non_cm_j_invariants} having degree $d\geq 2$, the modular curves have rank 0 by \cite[Theorem 3.1]{DEvHMZB2021}. This justifies the listed points are isolated except for $X_1(25)$ and $X_1(32)$. For those, we apply Theorem \ref{IsolThm} combined with \cite[Theorem 3]{DerickxVanHoeij2014} for $X_1(25)$ and \cite[Theorem 1.1]{NajmanVarivoda} for $X_1(32)$.

\subsection{Justifying Non-CM} If $x\in X_1(N)$ is a CM point with $j(x)$ of degree $d>1$, then by \cite[Theorem 6.2]{BC1}
\[
\deg(x) \geq d \cdot \varphi(N)/2.
\]
Thus all points in the table associated to $j$-invariants of degree greater than 1 are non-CM.

\section{On Stronger Polynomial Bounds}
Hindry and Silverman \cite[$\S3$]{HS99} ask whether there exists a constant $c$ such that for any non-CM elliptic curve $E$ defined over $F$ of degree at least 3, one has $\# E(F)_{\tors} \leq c \sqrt{[F:\Q] \log \log [F:\Q]}$. This strengthens Hypothesis 4 by assuming the existence of a single, absolute constant $c$, in addition to giving a better bound for $\# E(F)_{\tors}$ as a function of $[F:\Q]$. Moreover, this is essentially the strongest possible upper bound on the full torsion subgroup: For any non-CM elliptic curve $E$ over $\Q$, Breuer \cite[Theorem 2.1]{Bre10} has constructed a sequence of fields $F_n$ of degree $d_n \geq 3$ such that $\#E(F)_{\tors} \gg \sqrt{d_n \log \log d_n}$.

Here, we show that an affirmative answer to even the exponent bound suggested by Hindry and Silverman would imply Hypothesis 1. Thus if $C$ can be taken to be an absolute constant in Hypothesis 4, this would imply Hypothesis 1.

\begin{theorem}\label{LastResult}
    Suppose there exists an absolute constant $c$ such that for any non-CM elliptic curve $E$ defined over a number field $F$ of degree at least 3, one has $\exp E(F)_{\tors} \leq c \sqrt{[F:\Q] \log \log [F:\Q]}$. Then Hypothesis 1 holds.
\end{theorem}

\begin{proof}
   Suppose Hypothesis 1 fails for some $d \in \Z^+$. Then for arbitrarily large primes $p$, there exists a non-CM elliptic curve $E$ defined over a degree $d$ number field for which $\im \rho_{E,p^{\infty}}$ does not contain $\SL_2(\Z_{p})$. In particular, provided $p \geq 5$, this means the mod $p$ image does not contain $\SL_2(\Z/p\Z)$ for arbitrarily large primes $p$; see \cite[Lemma 2.12]{Greicius10}. By \cite[Remark 2.1]{GhateParent}, $\im \rho_{E,p}$ cannot have image in $\text{PGL}_2(\Z/p\Z)$ isomorphic to $A_4,S_4,$ or $A_5$ once $p>60d+1$. Thus by
    \cite[$\S2$]{serre72}, one of the following cases must occur for infinitely many primes $p$:
    \begin{enumerate}
        \item $\im \rho_{E,p}$ is contained in a Borel subgroup. In this case, $E$ attains a point of order $p$ in an extension of degree dividing $p-1$. Thus, there is a number field $F$ of degree dividing $d(p-1)$ such that $E(F)$ contains a point of order $p$.
        \item $\im \rho_{E,p}$ is contained in the normalizer of a split Cartan subgroup. Here, $E$ attains a $p$ isogeny over a quadratic extension, and so as above there is a number field $F$ of degree dividing $2d \cdot (p-1)$ such that $E(F)$ contains a point of order $p$.
         \item $\im \rho_{E,p}$ is contained in the normalizer of a non-split Cartan subgroup. Here, $E$ attains full $p$-torsion over an extension of degree at most $2(p^2-1)$. As shown in \cite[Theorem 3.1]{OddDegQCurve}, there is an elliptic curve $E'$ (geometrically isogenous to $E$) defined over a number field $F$ of degree at most $d\cdot 2p(p^2-1)$ which has an $F$-rational point of order $p^2$.
    \end{enumerate}
    In each case, this would contradict the assumption that for any $F$ of degree at least 3,
    \[
    \exp E(F)_{\tors} \leq c \sqrt{[F:\Q] \log \log [F:\Q]}.
    \qedhere
    \] 
\end{proof}

\bibliographystyle{amsplain}
\bibliography{bibliography}
\end{document}